\documentclass{amsart}

\allowdisplaybreaks

\usepackage{amssymb}

\theoremstyle{plain}
\newtheorem{theorem}{Theorem}[section]

\theoremstyle{definition}

\theoremstyle{remark}

\begin{document}

\title[Multiplicative functions commutable with $x^2 \pm xy + y^2$]
{Multiplicative functions commutable\\ with binary quadratic forms \(x^2 \pm xy + y^2\)}

\author{Poo-Sung Park}
\address{Department of Mathematics Education, Kyungnam University, Changwon, Republic of Korea}
\email{pspark@kyungnam.ac.kr}

\thanks{This work was supported by Kyungnam University Foundation Grant, 2019.}


\begin{abstract}
If a multiplicative function $f$ is commutable with a quadratic form $x^2+xy+y^2$, i.e., 
\[
f(x^2+xy+y^2) = f(x)^2 + f(x)\,f(y) + f(y)^2,
\]
then $f$ is the identity function. In other hand, if $f$ is commutable with a quadratic form $x^2-xy+y^2$, then $f$ is one of three kinds of functions: the identity function, the constant function, and an indicator function for $\mathbb{N}\setminus p\mathbb{N}$ with a prime $p$.
\end{abstract}

\maketitle

\section{Introduction}

In 2014, Ba\v{s}i\'c classified arithmetic functions $f$ satisfying 
\[
f(m^2+n^2) = f(m)^2 + f(n)^2
\]
for all positive integers $m$ and $n$. His result was a variation of Chung's work \cite{Chung}, which was inspired from Claudia Spiro's study about \emph{additive uniqueness sets} \cite{Spiro}. It is naturally generalized to studying arithmetic functions $f$ satisfying 
\[
f\big(Q(x_1, x_2, \dots, x_k)\big) = Q\big(f(x_1), f(x_2), \dots, f(x_k)\big)
\]
for various quadratic forms $Q$. After Ba\v{s}i\'c's work for $Q(x,y) = x^2 + y^2$, You et al. \cite{YCYS} and Khanh \cite{Khanh} studied about $Q(x,y) = x^2 + ky^2$.

The author extended Ba\v{s}i\'c's work to multiplicative functions commutable with sums of more than $2$ squares. That is, if a multiplicative function $f$ satisfy 
\[
f(x_1^2 + x_2^2 + \dots + x_k^2) = f(x_1)^2 + f(x_2)^2 + \dots + f(x_k)^2
\] for $k \ge 3$, then $f$ is uniquely determined to be the identity function \cite{Park}. 

Let $Q(x,y) = ax^2 + bxy + cy^2$ be a positive definite binary quadratic form with $a,b,c \in \mathbb{Z}$. The value $b^2-4ac$ is called \emph{discriminant} of $Q$. The discriminant of the smallest absolute value is $-3$ for $x^2 \pm xy+y^2$. So, it is a natural question to ask which multiplicative function $f$ satisfies the condition 
\[
f(x^2 \pm xy + y^2) = f(x)^2 \pm f(x)\,f(y) + f(y)^2.
\]
In this article, we classify such multiplicative functions.

\section{Results}

\begin{theorem}
If a multiplicative function $f:\mathbb{N}\to\mathbb{C}$ satisfies 
\[
f(x^2+xy+y^2) = f(x)^2+f(x)\,f(y)+f(y)^2,
\]
then $f$ is the identity function.
\end{theorem}

\begin{proof}
We will show that $f(n)=n$ for $1 \le n \le 28$ and use induction.

Note that $f(1)=1$ and $f(3) = 3$ with $x=y=1$. Since $f$ is multiplicative, the values of $f$ at powers of primes determine $f$.

If $n$ is not divisible by $3$, then $f(n^{2}) = f(n)^2$ from
\begin{align*}
f(3n^{2}) 
&= f(3)\,f(n^{2}) = 3f(n^{2}) \\
&= f(n^{2}+n\cdot n+n^{2}) = 3f(n)^2. 
\end{align*}
Thus, $f(4) = f(2)^2$, $f(16)=f(4)^2$, and $f(25)=f(5)^2$.

Since 
\begin{align*}
f(7) 
&= f(1^2+1\cdot2+2^2) = 1+f(2)+f(2)^2, \\
f(13)
&= f(1)^2+f(1)\,f(3)+f(3)^2 = 13, \\
f(21) 
&= f(3)\,f(7) = 3f(7) \\
&= f(1)^2+f(1)\,f(4)+f(4)^2 = 1+ f(2)^2 + f(2)^4, \\
f(39)
&= f(3)\,f(13) = 39 \\
&= f(2)^2 + f(2)\,f(5) + f(5)^2, \\
f(91)
&= f(7)\,f(13) = 13\,f(7)\\
&= f(5)^2 + f(5)\,f(6) + f(6)^2 = f(5)^2 + 3 f(2) f(5) + 9 f(2)^2,
\end{align*}
we can conclude that $f(n) = n$ for $n=2,4,16, 5, 7, 13$

Since
\begin{align*}
f(84) 
&= f(4)\,f(3)\,f(7) = 4\cdot3\cdot7 = 84 \\
&= f(2)^2+f(2)\,f(8)+f(8)^2 = 4+2f(8)+f(8)^2, \\
f(43)
&= f(1)^2+f(1)\,f(6)+f(6)^2 = 43, \\
f(129)
&= f(3)\,f(43) = 3\cdot43 = 129 \\
&= f(5)^2+f(5)\,f(8)+f(8)^2 = 25+5f(8)+f(8)^2,
\end{align*}
we can find $f(8)=8$.

Since $f(7)=7$ and 
\begin{align*}
f\!\left(3^2+3\cdot(2\cdot3)+(2\cdot3)^2\right) 
&= f(3)^2+f(3)\,f(2\cdot3)+f(2\cdot3)^2 = 7f(3)^2 \\
&= f(7\cdot3^2) = f(7)\,f(9),
\end{align*}
we obtain that $f(9)=9$.

The next prime is $11$. But we need to find $f(19)$ to determine $f(11)$. Note that $f(19) = f(2)^2+f(2)\,f(3)+f(3)^2 = 19$. Now, since
\begin{align*}
f(133)
&= f(7)\,f(19) = 7\cdot19 = 133 \\
&= f(1)^2+f(1)\,f(11)+f(11)^2 = 1+f(11)+f(11)^2, \\
f(247) 
&= f(13)\,f(19) = 13\cdot19 = 247 \\
&= f(7)^2+f(7)\,f(11)+f(11)^2 = 49+7f(11)+f(11)^2,
\end{align*}
we can find $f(11)=11$.

Note that
\begin{align*}
f(399)
&= f(3)\,f(7)\,f(19) = 3\cdot7\cdot19 = 399 \\
&= f(5)^2+f(5)\,f(17)+f(17)^2 = 25+5f(17)+f(17)^2, \\
f(427)
&= f(3)^2+f(3)\,f(19)+f(19)^2 = 427 \\
&= f(6)^2+f(6)\,f(17)+f(17)^2 = 36+6f(17)+f(17)^2.
\end{align*}
Thus, $f(17)=17$.

We have $f(23) = 23$ from
\begin{align*}
f(553)
&= f(7)\,f(79) = 7\left(f(3)^2+f(3)\,f(7)+f(7)^2\right) = 7\cdot79 = 553 \\
&= f(1)^2+f(1)\,f(23)+f(23)^2 = 1+f(23)+f(23)^2, \\
f(579)
&= f(3)\,f(193) = 3\left(f(7)^2+f(7)\,f(9)+f(9)^2\right) = 3\cdot193 = 579 \\
&= f(2)^2+f(2)\,f(23)+f(23)^2 = 4+2f(23)+f(23)^2.
\end{align*}

Note that
\begin{align*}
f(27)
&= f(3)^2+f(3)\,f(3)+f(3)^2 = 27.
\end{align*}

From the above results, it appears that $f(n)=n$ for $1 \leq n \leq 28$.

Now, consider $f(n)$ for $n \ge 29$. We divide two cases: $n = 2k+1$ and $n=2k$.

Note that 
\begin{align*}
&(2k+1)^2 + (2k+1)(k-3) + (k-3)^2 \\
&= (2k-3)^2 + (2k-3)(k+2) + (k+2)^2
\end{align*}
when $k > 3$. Thus, if we assume that $f(m)=m$ for all $m < n = 2k+1$, we can write a functional equation
\begin{align*}
&f(2k+1)^2 + f(2k+1)(k-3) + (k-3)^2\\
&= (2k-3)^2 + (2k-3)(k+2) + (k+2)^2
\end{align*}
for $f(2k+1)$ by induction hypothesis and we obtain
\[
f(2k+1) = 2k+1 \qquad\text{ or }\qquad f(2k+1) = -3k+2.
\]

In other hand, since
\begin{align*}
&f\!\left((2k+1)^2 + (2k+1)(k-10) + (k-10)^2\right) \\
&= f\!\left((2k-11)^2 + (2k-11)(k+5) + (k+5)^2\right)
\end{align*}
when $k > 10$, we obtain
\[
f(2k+1) = 2k+1 \qquad\text{ or }\qquad f(2k+1) = -3k+9.
\]
Therefore, the solution satisfying both equalities simultaneously is that $f(n) = f(2k+1) = 2k+1$.

Similarly, from
\begin{align*}
&(2k)^2+(2k)(k-7)+(k-7)^2 \\
&= (2k-8)^2+(2k-8)(k+3)+(k+3)^2
\end{align*}
with $k > 7$ we obtain that
\[
f(2k) = 2k \qquad\text{ or }\qquad f(2k) = -3k+7
\]
if we assume that $f(m)=m$ for $m < n = 2k$. Also, from
\begin{align*}
&(2k)^2+(2k)(k-14)+(k-14)^2 \\
&= (2k-16)^2+(2k-16)(k+6)+(k+6)^2
\end{align*}
with $k > 14$ we obtain that
\[
f(2k) = 2k \qquad\text{ or }\qquad f(2k) = -3k+14.
\]
Therefore, we conclude that $f(n) = f(2k) = 2k$.
\end{proof}

\begin{theorem}
A multiplicative function $f:\mathbb{N}\to\mathbb{C}$ satisfies 
\[
f(x^2-xy+y^2) = f(x)^2-f(x)\,f(y)+f(y)^2
\]
if and only if $f$ is one of the following:
\begin{enumerate}
\item the identity function $f(n)=n$
\item the constant function $f(n)=1$
\item function $f_p$ defined by 
\[
f_p(n) = 
\begin{cases}
0 & p \mid n \\
1 & p \nmid n
\end{cases}
\]
for some prime $p \equiv 2 \pmod{3}$.
\end{enumerate}
\end{theorem}

\begin{proof}
It is trivial that the identity function and the constant function satisfy the functional equation. Let us consider the third case $f_p$.

It is known that $p \equiv 2 \pmod{3}$ if and only if $p$ cannot be represented as $x^2-xy+y^2$ \cite{Cox}. Assume that $n=x^2-xy+y^2$. Then $n=(x+\omega y)(x+\overline{\omega}y)$ with $\omega = (-1+\sqrt{-3})/2$ is the factorization in $\mathbb{Z}[\omega]$ a PID.  If $n$ is divisible by $p$, both $x$ and $y$ are divisible by $p$, since $p$ is an inert prime in $\mathbb{Z}[\omega]$. Thus, the function $f_p$ works well.

Now let us prove ``only if'' part. Note that
\[
f(n^2) = f(n^2 - n \cdot n + n^2) = f(n)^2 - f(n)\,f(n) + f(n)^2 = f(n)^2.
\]
We have that $f(1)=1$. From the equalities
\begin{align*}
f(3)
&= f(1)^2-f(1)\,f(2)+f(2)^2 = 1-f(2)+f(2)^2 \\
f(7)
&= f(1)^2-f(1)\,f(3)+f(3)^2 = 1-f(3)+f(3)^2 \\
&= f(2)^2-f(2)\,f(3)+f(3)^2 = f(2)^2-f(2)\,f(3)+f(3)^2,
\end{align*}
there are three cases:
\[
\begin{array}{llllll}
f(1)=1, \quad &f(2)=0, \quad &f(3)=1, \quad &f(4)=0, \quad &f(6)=0, \quad &f(7)=1; \\
f(1)=1, \quad &f(2)=1, \quad &f(3)=1, \quad &f(4)=1, \quad &f(6)=1, \quad &f(7)=1; \\
f(1)=1, \quad &f(2)=2, \quad &f(3)=3, \quad &f(4)=4, \quad &f(6)=6, \quad &f(7)=7.
\end{array}
\]

Since
\begin{align*}
f(1-n+n^2) &= f(1^2-1\cdot n+n^2) \\ &= 1-f(n)+f(n)^2
\intertext{and}
f(1-n+n^2) &= f\!\left((n-1)^2-(n-1)n+n^2\right) \\ &= f(n-1)^2-f(n-1)\,f(n)+f(n)^2,
\end{align*}
we have that
\[
f(n-1)^2-f(n-1)\,f(n) = 1-f(n)
\]
or
\[
\big(f(n-1)-f(n)+1\big)\big(f(n-1)-1\big) = 0.
\]
Thus, it yields a condition
\[
f(n-1) = 1 \qquad \text{or} \qquad f(n) = f(n-1)+1.\tag{$\ast$} \label{eq:f(n-1)}
\]
So, if $f(2)=2$, then $f(3)=3$ and thus $f(4)=4$, and so forth. We obtain the identity function $f(n)=n$ when $f(2)=2$.

If $f(2)=0$, then we have $f(3) = f(5) = f(7) = 1$ and $f(4) = f(6) = 0$. From
\begin{align*}
f\!\left(2^2-2\cdot(2k)+(2k)^2\right)
&= f(2)^2-f(2)\,f(2k)+f(2k)^2 = f(2k)^2 \\
&= f(4-4k+4k^2) = f(4)\,f(1-k+k^2) = 0
\end{align*}
we deduce that $f(2k)=0$ for $k \ge 1$. Since $f(2k+1)=1$ by condition ($\ast$), $f(2)=0$ yields a sequence alternating $1$ and $0$. That is, $f=f_2$.

Now, the condition $f(1)=f(2)=f(3)=f(4)=f(6)=f(7)=1$ remains. If $f(n)=a$ for some $a \in \mathbb{C} \setminus \{ 1, 0, -1, -2, \dots\}$, then $f(m) \ne 1$ for all $m > n$. But, since $1=f(2)=f(2^2)=f(2^4)=\dots=f(2^{2^N})$ for sufficiently large $N$, it is a contradiction. So, we can deduce that $f(n)$ can have only integers $\le 1$.

Suppose that $s$ is the smallest integer such that $f(s)=0$. If there exist no such $s$, then $f$ is a constant function $f(n)=1$.

Since $f$ is multiplicative and $f(n^2)=f(n)^2$, we can say that $s=p^{2k-1}$ with prime $p$ and positive integer $k$. Note that
\begin{align*}
&f\!\left( (p^{2k})^2 - p^{2k}p^{2k-1} + (p^{2k-1})^2 \right) \\
&= f(p^{2k})^2 - f(p^{2k})\,f(p^{2k-1}) + f(p^{2k-1})^2 = f(p^{2k})^2\\
&= f\!\left( (p^{2k-1})^2 ( p^2 - p + 1 ) \right) = f(p^{2k-1})^2 f(p^2-p+1) = 0.
\end{align*}
Thus, $f(p^{2k}) = f(p^k)^2 = 0$. By the minimality of $s = p^{2k-1}$, we can deduce that $k=1$. That is, $s$ is the prime $p$ itself.

Then, we obtain $f(p\ell)=0$ for any positive integer $\ell$, since
\begin{align*}
&f\!\left(p^2-p(p\ell)+(p\ell)^2\right) \\
&= f(p)^2 - f(p)\,f(p\ell)+f(p\ell)^2 = f(p\ell)^2 \\
&= f\!\left(p^2(1-\ell+\ell^2)\right) = f(p)^2 f(1-\ell+\ell^2) = 0.
\end{align*}
That is, we can conclude that
\[
f(n)=f_p(n)=0\text{ when $n$ is a multiple of $p$} 
\]
and $f(p\ell+1)=1$ by condition ($\ast$).

Now, let $n$ be a positive integer with $p \nmid n$. Then, there exists an integer $m$ such that $n m \equiv 1 \pmod{p}$ and $(n,m)=1$. Letting $nm = p\ell+1$, we obtain 
\[
1 = f(p\ell+1) = f(nm) = f(n)\,f(m).
\]
Since $f$ can have only integers $\le 1$, we can conclude that 
\[
f(n)=\pm1 \text{ if }p \nmid n. \tag{$\ast\ast$}\label{eq:f(n)=+-1}
\] 

Now let us characterize the prime $p$. If $p$ can be represented as $x^2-xy+y^2$, then $0 = f(p) = f(x)^2-f(x)\,f(y)+f(y)^2$. But, this never happen since $f(x)$ and $f(y)$ are $\pm1$. Hence, $p \equiv 2 \pmod{3}$.

If $f(n)=-1$ for $n \le p-2$, then $f(n+1)=0$ with $n+1 \le p-1$ by ($\ast$). But, this is impossible by \eqref{eq:f(n)=+-1}. Thus, if $f(n)=-1$ for $n \le p-1$, then $n=p-1$. In this case, $f(d) = -1$ for a proper divisor $d$ of $p-1$ unless $p$ is a Fermat prime. Thus, it is a contradiction. If $p = 2^{2^r}+1$, then $-1 = f(p-1) = f(2^{2^r}) = f(2^{2^{r-1}})^2$, which is impossible for $f(2^{2^{r-1}}) = \pm1$. So, we can conclude that
\[
f(1) = f(2) = f(3) = \dots = f(p-1) = 1 \text{ and } f(p) = 0.
\]

Similarly, suppose that $f(n)=-1$ for some $n$. If $p(\ell-1)+1 \le n \le p\ell-2$ with $\ell \ge 2$, then $f(n+1)=0$ with $p(\ell-1)+2 \le n+1 \le p\ell-1$ by \eqref{eq:f(n-1)}. But, this is a contradiction by \eqref{eq:f(n)=+-1} since $n+1$ is not a multiple of $p$. Hence, if $f(n)=-1$ with $p(\ell-1)+1 \le n \le p\ell-1$, then $n = p\ell-1$. Then, since $n-1$ and $n$ are relatively prime and $(n-1)n = (p\ell-2)(p\ell-1) \equiv 2 \not\equiv -1 \pmod{p}$, we can deduce a contradictary equality
\begin{align*}
f\big( (n-1)n \big) 
&= f(n-1) f(n) = -1 \\
&= f\big( (p\ell-2)(p\ell-1) \big) = 1.
\end{align*}

Therefore, we can conclude that $f(n) = f_p(n) = 1$ when $n$ is not divisible by $p$.
\newline

\end{proof}


\begin{thebibliography}{99}

\bibitem{Basic} B. Ba\v{s}i\'c, 
	Characterization of arithmetic functions that preserve the sum-of-squares operation, 
	\textit{Acta Math. Sin.} \textbf{30} (2014), 689--695.

\bibitem{Cox} D. A. Cox, 
	Primes of the form $x^2+ny^2$, Wiley, 1989.

\bibitem{Chung} P. V. Chung, 
	Multiplicative functions satisfying the equation $f(m^2 +n^2) = f(m^2)+f(n^2)$, 
	\textit{Math. Slovaca} \textbf{46} (1996), 165--171.

\bibitem{Khanh} B. M. M. Khanh, 
	On the equation $f(n^2 + D m^2) = f(n)^2+ D f(m)^2$, 
	\textit{Annales Univ. Sci. Budapest. Sect. Comp.} \textbf{46} (2017), 123--135.

\bibitem{Park} P.-S. Park, 
	On $k$-additive uniqueness of the set of squares for multiplicative functions, 
	\textit{Aequat. Math.} \textbf{92}, No. 3 (2018), 487--495.

\bibitem{Spiro} C. A. Spiro, 
	Additive uniqueness sets for arithmetic functions, 
	\textit{J. Number Theory} \textbf{42} (1992), 232--246.

\bibitem{YCYS} L. You, Y. Chen, P. Yuan, J. Shen, 
	A characterization of arithmetic functions satisfying $f(u^2+kv^2)=f^2(u)+kf^2(v)$, 
	arXiv:1606.05039 [math.NT].
\end{thebibliography}
\end{document}